\documentclass[11pt]{amsart}

\usepackage{geometry}                % See geometry.pdf to learn the layout options. There are lots.
\usepackage{amssymb}
\usepackage{epstopdf}

\usepackage{enumerate}
\usepackage[all]{xy}
\usepackage{amsmath}
\usepackage{graphicx}

\usepackage{xcolor}

\newtheorem{theorem}{Theorem}%[section]
\newtheorem{proposition}[theorem]{Proposition}
\newtheorem{lemma}[theorem]{Lemma}
\newtheorem{corollary}[theorem]{Corollary}

\newtheorem{definition}[theorem]{Definition}
\newtheorem{example}[theorem]{Example}

\def\R{\mathbb R}

\def\tc{{\mathsf{TC}}}

\title[Autonomous motion in changing environment]{Autonomous motion in changing environment, fibrations and reaction mechanisms}
\author{Michael Farber}
\address{Michael Farber, Queen Mary University of London and AIMATH RESEARCH, London, UK}
\email{m.farber@qmul.ac.uk}
\author{Stefan Kurz}
\address{Stefan Kurz, Robert Bosch GmbH, Corporate Research}
\email{stefan.kurz2@de.bosch.com}
\author{Mathias Pillin}
\address{Mathias Pillin, Robert Bosch GmbH, Bosch Mobility}
\email{mathias.pillin@de.bosch.com}

\date{24 June 2026}                                           % Activate to display a given date or no date

\begin{document}
\begin{abstract} In this paper we develop further the formalism of fibrations of configuration spaces as a tool for modelling motion of
autonomous systems in variable environments. 
We analyse the situations when the external conditions may change during the motion of the system and analyse two possibilities: (a) when the behaviour of the external conditions is known in advance; and (b) when the future changes of the external conditions are unknown but we can measure the current state and the current velocity of the external conditions, at every moment of time. 
We prove that 
in the case (a) 
the complexity of the motion algorithm is the same as in the case of constant external conditions. 
In case (b) we introduce a new concept of {\it a reaction mechanism} which allows to take into account unexpected and unpredictable changes in the environment. A reaction mechanism is mathematically {\it an infinitesimal lifting function} on a fibre bundle, a nonlinear generalisation of the classical concept of an Ehresmann connection. 
We illustrate these notions by examples which show that nonlinear infinitesimal lifting function (reaction mechanisms) appear naturally, are inevitable and ubiquitous.  
\end{abstract}
\maketitle
\section{Modeling autonomous systems moving in varying environments} 
\subsection{} Consider an autonomous system  $M$, {\it \lq\lq the machine\rq\rq},  moving in a variable environment. 
As an example, one may think of a robot or an autonomous vehicle \cite{SSS} operating in an area populated with some other moving objects, robots and humans.
As another example, we may mention a flock of fully automated drones 
cruising in space with the enemy drones and airplanes forming the moving environment. 
Yet another example is an autonomous submarine 
navigating in waters populated with moving mines. 
In this paper we analyse the question of how to build a motion algorithm for our autonomous system $M$ so that it achieves the goal regardless of the behaviour and unpredictable changes of the environment? 
We shall answer this question by employing the language of differential and algebraic topology and developing the mathematical concepts  
of {\it infinitesimal motion algorithm} or {\it infinitesimal reaction mechanisms}. 

The notion of infinitesimal reaction mechanism is a nonlinear generalisation of the classical notion of an Ehresmann connection; it allows selection of velocity of the autonomous system $M$ as a function of the velocity of changes in the environment. 

%
%We show that motion algorithms for autonomous systems are mathematically generalisations of connections on fibre bundles. We introduce the notion of {\it \lq\lq a reaction mechanism\rq\rq},  it is a mathematical structure 
\subsection{} The notion of a configuration space is well-known in modern robotics \cite{LP}. This notion describes the space of all admissible states of the system moving in a stationary environment. The configuration space can be viewed as a topological space equipped with some additional structures (such as a smooth manifold structure, a Riemannian metric etc). 

\subsection{}\label{sec:13} In the recent papers \cite{CFW21, CFW22} the concept of a configuration space was extended to the case of varying external conditions, and a new {\it parametrised approach} was developed, see also \cite{FW23}. In this theory, instead of a single configuration space, one studies the whole family of configuration spaces parametrised by the configuration space of external conditions. In this approach we have a fibration 
\begin{eqnarray}\label{p}
p: E\to B,
\end{eqnarray}
where the base $B$ is the configuration space of the environment (the external conditions) and the total space $E$ is the configuration space of the compound system consisting of the machine $M$ and the environment. More specifically, the points $b\in B$ represent various states of the environment and for each 
$b\in B$ the fibre $E_b= p^{-1}(b)\subset E$ is the configuration space of the machine $M$ constrained by the external conditions $b$. Thus, the total space $E$ 
is the disjoint union $\bigsqcup_{b\in B}E_b$ of the configuration spaces of our system $M$ which are {\it \lq\lq related\rq\rq}\ by means of the topology of the space $E$. We shall refer to (\ref{p}) as {\it a fibration of configuration spaces}. 

In many practically interesting situations the projection (\ref{p}) is {\it a locally trivial bundle}. This means that the  
space of external conditions $B$ admits an open cover $\{U_i\}_{i\in J}$ with the property that each 
preimage $p^{-1}(U_i)$ is homeomorphic to the product $X\times U_i$ (where $X$ is the fibre). More precisely, this means that there is a continuous map $F_i: p^{-1}(U_i)\to X$ such that the map 
$$p^{-1}(U_i)\to X\times U_i, \quad e\mapsto (F_i(e), p(e)),\quad \mbox{where}\quad e\in p^{-1}(U_i), \quad i\in J,$$
is a homeomorphism. We see that the spaces of configurations $E_{b_1}=p^{-1}(b_1)$ and $E_{b_2}=p^{-1}(b_2)$ living under close enough external conditions $b_1\sim b_2$ can be naturally identified.

As a specific important example, the authors studied in \cite{FW23}
the algorithms of motion of multiple objects in space avoiding collisions with multiple obstacles. 

\subsection{}\label{sec:14} Next we recall the notion of {\it a parametrised motion planning algorithm}, see \cite{CFW21}, \cite{CFW22}, \cite{FW23}. 
Such an algorithm takes as input pairs of configurations $(e, e')$ living under the same external conditions (i.e. such that $p(e)=p(e')\in B$) and produces as output a continuous motion of the system $\gamma: [0,1]\to E$
with the properties $\gamma(0)=e$, $\gamma(1) =e'$ and, moreover, $p(\gamma(t))=p(e)=p(e')\in B$ for any $t\in [0,1]$; the latter property means that the motion of the system is performed under the {\it constant} external conditions. 
\begin{figure}[h]
\begin{center}
\includegraphics[scale=0.3]{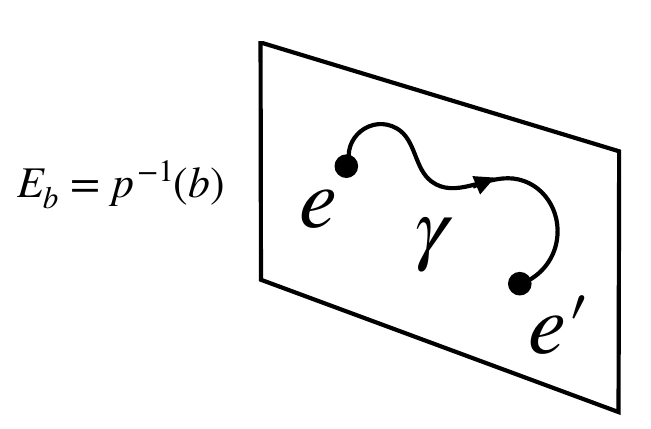}
\caption{Parametrized motion planning algorithm}\label{fig:fiber1}
\end{center}
\end{figure}
Hence, we see that while parametrised motion planning algorithms work for varieties of external conditions, they are based on the assumption that {\it the external conditions do not change during the implementation of the motion}. This is the main distinction with the novel approach which we develop in this paper. 

\subsection{}\label{sec:15} Next we recall the notion of parametrised topological complexity \cite{CFW21}, \cite{CFW22}, \cite{FW23}. 
Given a locally trivial bundle $p: E\to B$ with fibre $X$, we denote by $E^2_B \subset E\times E$ the subspace consisting of all pairs $(e, e')\in E\times E$ with $p(e)=p(e')\in B$. 
Besides, we denote by $E^I_B\subset E^I$ the subspace of the path-space consisting of all continuous paths 
$\gamma: I\to E$ such that the path $p\circ \gamma: I\to B$ is constant; here $I=[0,1]$ denotes the unit interval. The evaluation map
\begin{eqnarray}\label{pi}
\Pi: E^I_B \to E^2_B, \quad \mbox{where}\quad \Pi(\gamma) =(\gamma(0), \gamma(1)),
\end{eqnarray}
is a locally trivial fibration. Its fibre over a pair $(e, e')\in E^2_B$ is the space of all 
paths in the fibre starting at $e$ and ending at $e'$; this space is homotopy equivalent to $\Omega X$, the space of based loops in $X$. 

{\it A parametrised motion planning algorithm} (as described in \S \ref{sec:14}) is a section
\begin{eqnarray}\label{sec}
s: E^2_B\to E^I_B
\end{eqnarray} 
of the fibration 
(\ref{pi}). The following definition gives a natural measure of complexity of a parametrised motion planning algorithm:

\begin{definition}\label{def1} Let $p: E\to B$ be a locally trivial bundle with the base $B$ and the fibre $X$ being metrisable separable ANR's (Absolute Neighbourhood Retracts). 
The parametrised topological complexity $\tc[p:E\to B]$ is the smallest integer $k\ge 0$ such that there exists
a partition
\begin{eqnarray}\label{part}
E^2_B= F_0\sqcup F_1\sqcup \dots\sqcup F_k
\end{eqnarray}
with the property that over each set $F_i$ there exists a continuous section $s_i:F_i\to E^I_B$ of the fibration (\ref{pi}), where $i=0, 1, \dots, k$. The sections $s_0, \dots, s_k$ determine a globally defined section $s: E^2_B\to E^I_B$ (i.e. a parametrised motion planning algorithm) by the rule $s|_{F_i}=s_i$. 
\end{definition}

The assumptions that the base $B$ and the fibre $X$ are  metrisable separable ANR's is satisfied in all practically interesting situations. However, if these assumptions are not valid, then instead of partitions (\ref{part}) one speaks about open covers 
\begin{eqnarray}\label{cov}
E^2_B= U_0\cup U_1\cup \dots\cup U_k,
\end{eqnarray}
where the minimal value of $k$ equals $\tc[p:E\to B]$. We refer to Proposition 4.7 from \cite{CFW21} 
which gives equivalence of results based on open covers and partitions. This result from \cite{CFW21} is based on the work \cite{GC}. 

In some situations the parametrised topological complexity $\tc[p:E\to B]$ equals the usual topological complexity $\tc(X)$ of the fibre $X$. 
This happens in the case of the trivial bundle and also for the principal bundles $p: E\to B$, see \cite{CFW21}, \cite{CFW22}, \cite{FW23}. 
An example where $\tc[p:E\to B]>\tc(X)$ is the Fadell - Neuwirth bundle 
\begin{eqnarray}\label{fn}
p: F(\R^d, n+m)\to F(\R^d, m)
\end{eqnarray} 
studied in \cite{CFW21} and \cite{CFW22}. Here the symbol $F(\R^d, n)$ stands for the configuration space of $n$ pairwise distinct points in $\R^d$. The Fadell - Neuwirth bundle (\ref{fn})
is the fibration of configuration spaces (\ref{p}) describing collision free motion of $n$ point-size robots in the presence of $m$ point-size obstacles. 
 The parametrised topological complexity of the bundle 
(\ref{fn}) equals $2n+m-1$ for $d\ge 3$ odd and it equals $2n+m-2$ for $d\ge 2$ even. 
We refer to the original articles for details. 

%Below in this paper we shall analyse configuration spaces of \lq\lq thick\rq\rq\ robots and relevant motion algorithms. 

\section{Lifting functions and reaction mechanisms}\label{sec:2}

In this section we introduce terminology and give literature references. 

\subsection{} Consider an autonomous system $M$ moving in a variable environment. We assume that the environment is changing and this forces our system $M$ to move such that (a) $M$ remains consistent with the motion of the environment and (b) $M$ implements its ultimate goal. 
As an example, one can think of a robot $M$ moving in an environment occupied by some other moving objects (robots, machines, humans); the task (a)  in this case means that $M$ avoids collisions with the others. 

%We mainly have in mind the situation with a driverless car embedded into the moving traffic: the car must stay clear of the other moving vehicles and pedestrians 

{\it A reaction mechanism} is a device that solves the task (a) mentioned above. 

In many practical situations the task (a) has priority over (b), i.e. the motion algorithm must first ensure safety and then, if possible, move the controlled system to the desired state. In other words, we should study the variety of reaction mechanisms and, if possible, select the one implementing the task (b). 

The problem can be stated 
using the language of fibrations $p:E\to B$ of configuration spaces (\ref{p}) as follows. The present state of the system $M$ is described by the point $e\in E$ and the present state of the external conditions is described by the point $b=p(e)\in B$. The external conditions move; their motion is represented by a continuous curve $\gamma: [0, 1]\to B$ satisfying $\gamma(0)=b$. A reaction mechanism $L$ associates with 
the curve $\gamma$ and the point $e\in E$ {\it a lifted curve} $\tilde \gamma: [0,1]\to E$ satisfying $\tilde \gamma(0)=e$ and $p(\tilde\gamma(t))=\gamma(t)$ for all $t\in [0,1]$. In other words, $L$ is the map $\gamma\mapsto \tilde \gamma$. 

We shall denote by $I=[0,1]$ the unit interval and by $B^I$ and $E^I$ the spaces of continuous maps $I\to B$ and $I\to E$ correspondingly. We equip the path-spaces $B^I$ and $E^I$ with the compact-open topology, \cite{Sp}. We shall consider also the space $E\times_B B^I\subset E\times B^I$, which is defined by
\begin{eqnarray}\label{bbe}
E\times_B B^I\, =\, \{(e, \gamma)\in E\times B^I; p(e)=\gamma(0)\}.
\end{eqnarray} 
%which is defined as the subspace consisting of the pair $(e, \gamma)\in E\times B^I$ satisfying 
%$p(e)=\gamma(0)$. 
The projection 
\begin{eqnarray}
\pi: E^I\to E\times_B B^I
\end{eqnarray}
associates with any path $\tilde \gamma\in E^I$ the pair $(e, \gamma)\in E\times_B B^I$ where $e=\tilde \gamma(0)$ and $\gamma=p\circ \tilde\gamma$. 

\begin{definition} \label{liftingf}
A lifting function for the map $p:E\to B$ is a continuous map 
\begin{eqnarray}\label{eq:9a}
L: E\times_B B^I\to E^I, \quad L(e, \gamma)=\tilde\gamma,
\end{eqnarray} 
satisfying $\pi\circ L={\rm id}$, i.e. $e=\tilde \gamma(0)$ and $\gamma=p\circ \tilde\gamma$. 
\end{definition}

This definition appears in \cite{Sp}, chapter 2. 

\begin{corollary}
A lifting function $L$ for the fibration of configuration spaces $p: E\to B$ provides a reaction mechanism for the controlled system $M$. 
The motion of the compound system produced by the reaction mechanism 
$$\tilde \gamma(t) =L(e, \gamma)(t), \quad t\in [0,1],$$ 
is consistent with the variable external conditions changing as described by the path $\gamma(t)$. 
\end{corollary}

\subsection{} At an intermediate moment of time $t\in (0,1]$ the automated system {\it \lq\lq knows\rq\rq\ } only a part of the trajectory $\gamma$ of the external conditions, 
namely only $\gamma|_{[0,t]}$ can be measured by the sensing and observation block of the autonomous system by time $t$. We see that the reaction mechanism $L$ must possess 
some additional properties to be able to process the information about the external conditions and at the same time continuously produce the motion of the controlled system. The motion mechanism of this type are {\it \lq\lq incremental\rq\rq}; they will be considered and formalised in \S \ref{sec:ilf}. 

\subsection{} The existence of a Reaction Mechanism (i.e. a lifting function) imposes certain restrictions on the map $p: E\to B$. 

\begin{theorem}[See \cite{Sp}, Chapter 2, \S7, Theorem 8] A map $p: E\to B$ admits a lifting function if and only if is a fibration in the sense of Hurewicz. 
\end{theorem}

%Recall (cf. \cite{Sp}, chapter 2, \S 2) that a continuous map $p:E\to B$ is a fibration in the sense of Hurewicz if 
%for any commutative diagram of the form
%$$
%\xymatrix{X\ar[r]^{ f} \ar[d]_{i}& E\ar[d]^p\\
%X\times I\ar[r]_{F}& B
%}
%$$
%where $i:X\to X\times I$ is the inclusion $i(x)=(x, 0)$, there exists a homotopy $H: X\times I\to E$ such that 
%$p\circ H=F$ and $H\circ i=f$. 

We also note the following result:

\begin{theorem}[See \cite{Sp}, chapter 2, \S 7, Theorem 13] If $B$ is Hausdorff and paracompact then 
any locally trivial map $p:E\to B$ is a fibration in the sense of Hurewicz and hence it admits a lifting function. 
\end{theorem}

\section{The case when variation of external conditions is known in advance}\label{sec:3}

\subsection{} In this section we shall analyse the situation when an autonomous system $M$ moves in variable external conditions and the external conditions change during the motion of our system. 
We shall assume that
 {\it the variation of the external conditions is known in advance. } We show that in this case the complexity of the motion algorithms equals the complexity of the motion algorithms under {\it constant} external condition, i.e. it equals the parametrised topological complexity $\tc[p: E\to B]$. 
 The main results of this section are Theorem \ref{thm:8} and Theorem \ref{comp}; they generalise  an earlier result \cite{FGY}.
 
\subsection{} In practical situations the assumption about the variation of external conditions being known in advance can be justified if the system employs a closed loop algorithm involving continuous observations under the changes of the external conditions as well as {\it \lq\lq the block of predictions\rq\rq\ }
 suggesting the most likely trajectory of the motion of external conditions. The error of the obtained prediction will be small if the time interval is short; in that case the predicted motion of the external conditions will have small deviation from their real motion.

\subsection{} \label{subsec:32} We shall use terminology introduced earlier in \S \ref{sec:14}, \S \ref{sec:15}, \S \ref{sec:2}. We shall assume in this section that the map 
$p:E\to B$ is a locally trivial bundle with fibre $F$. We shall also assume that the base $B$ is Hausdorff and paracompact; these assumptions are satisfied in typical applications we have in mind. 

Recall that parametrized topological complexity $\tc[p:E\to B]$ is defined as the sectional category of the fibration 
(\ref{pi}). 
A generalisation of this notion to the situation with varying external conditions can be described as follows. 
Consider the space 
$$B^I\times_{B^2}E^2\subset B^I\times E^2= B^I\times E\times E$$ consisting of all triples $(\gamma, e, e')$, 
where $\gamma\in B^I$ and $e, e'\in E$,
satisfying 
$$\gamma(0)=p(e)\quad\mbox{and}\quad  \gamma(1)=p(e').$$ 
Consider also the map
\begin{eqnarray}\label{tildea2d}
\tilde \Pi: E^I \to B^I\times_{B^2} E^2,
\end{eqnarray}
where for $\tilde \gamma\in E^I$ one has 
$$
\tilde \Pi(\tilde \gamma) = (\gamma, e, e')
\quad \mbox{with}\quad \gamma =p\circ \tilde \gamma\in B^I, \quad e=\tilde\gamma(0)\quad\mbox{and} \quad e'=\tilde\gamma(1).
$$
\begin{lemma}\label{lm:6}
The map (\ref{tildea2d}) is a Hurewicz fibration, and its fibre is homotopy equivalent to the based loop space $\Omega F$, where $F$ is the fibre of the original bundle $p:E\to B$. 
\end{lemma}
A proof of Lemma \ref{lm:6} is given in the Appendix. 

A section of the bundle $\tilde \Pi$  (not necessarily continuous),
$$s: B^I\times_{B^2} E^2 \to E^I$$  can be viewed as {\it a motion algorithm} for our system working under variable external conditions. 
This algorithm takes as {\it input} triples $$(\gamma, e, e')\in B^I\times_{B^2} E^2$$ and produces as {\it  output} a motion of the compound system 
$$t\mapsto \tilde \gamma(t)= s(\gamma, e, e')(t)\in E, \quad t\in I,$$ which (1) is consistent with the motion of external conditions $\gamma$, (2) it starts at the prescribed state $e=\tilde\gamma(0)$ and (3) it ends at 
the prescribed state $e'=\tilde\gamma(1)$.

\begin{definition}
Let $\widetilde\tc[p:E\to B]$ denote the sectional category of the bundle (\ref{tildea2d}). In more detail, $\widetilde\tc[p:E\to B]$ is the minimal integer $r\ge 0$ such that the space 
$B^I\times_{B^2}E^2$ admits an open cover of cardinality $r+1$, 
$$
B^I\times_{B^2}E^2= W_0\cup W_1\cup \dots\cup W_r,
$$
such that each open set $W_j$ has a continuous section $s_j: W_j\to E^I$ of $\tilde \Pi$, 
where $j=0, 1, \dots, r$.
\end{definition}

In this approach information about the changes in external conditions $\gamma$ is part of the input of the motion algorithm; in other words, we assume that the automated system knows 
$\gamma$ before it starts its motion. 
The integer $\widetilde\tc[p:E\to B]$ can be viewed as a measure of complexity of the motion planning problem under variable external conditions.

\begin{theorem}\label{thm:8}
For any locally trivial bundle $p:E\to B$ with Hausdorff and paracompact base $B$ one has the equality
\begin{eqnarray}\label{in:11}
\widetilde\tc[p:E\to B]= \tc[p:E\to B].
\end{eqnarray}  
\end{theorem}
\begin{proof} The bundles $\Pi$ and $\tilde\Pi$ appear in the commutative diagram
$$
\xymatrix{
E^I_B\ar[r]^{\subset}\ar[d]_{\Pi}&E^I\ar[d]^{\tilde\Pi}\\
E^2_B\ar[r]^{i\hskip 0.6cm}_{\subset\hskip 0.6cm}&B^I\times_{B^2}E^2,
}
$$
where the map $i$ is given by $i(e, e')=(\gamma, e, e')$ for $(e, e')\in E^2_B$ with $\gamma\in B^I$ being the constant path at the point $p(e)=p(e')\in B$. 
Both horizontal maps in this diagram are injective and $\Pi$ is the restriction of $\tilde \Pi$ onto a subset of the base, i.e. 
$$E^I_B=\tilde\Pi^{-1}(i(E^2_B)).$$
This gives an inequality
$\widetilde\tc[p:E\to B]\ge \tc[p:E\to B].$ 

To show that it is in fact an equality we note that the subset $$i(E^2_B)\subset B^I\times_{B^2} E^2$$ is a strong deformation retract. Fix a lifting function 
$L: E\times_B B^I \to E^I$ for the bundle $p:E\to B$. Define the deformation 
\begin{eqnarray}
h_\tau: B^I\times_{B^2}E^2\to B^I\times_{B^2}E^2, \quad \tau\in [0,1],
\end{eqnarray}
as follows. For $\gamma\in B^I$ and $e, e'\in E$ with $p(e)=\gamma(0)$ and $p(e')=\gamma(1)$ we define 
\begin{eqnarray}
h_\tau(\gamma, e, e') = (\gamma_\tau, e, e'_\tau),\end{eqnarray}
where
\begin{eqnarray}
\gamma_\tau(t) = \gamma((1-\tau)t)\quad \mbox{and}\quad e'_\tau=L(e', \overline\gamma)(\tau)
\end{eqnarray}
with $\overline \gamma(t)=\gamma(1-t)$. Here $\overline \gamma$ is the inverse path $\gamma$ and $\tilde \gamma= L(e', \overline\gamma)$ is its lift with the initial condition
$e'\in E$. This construction is illustrated on Figure \ref{fig:reverse}. Note that 
$$p(e)=\gamma_\tau(0)=\gamma(0)\quad\mbox{ and}\quad  p(e'_\tau)= \overline\gamma(\tau)=\gamma(1-\tau)=\gamma_\tau(1).$$ 

We see that the map $h_0$ is the identity map $B^I\times_{B^2}E^2\to B^I\times_{B^2}E^2$ and the map 
$h_1$ takes its values in the subspace $i(E^2_B)\subset B^I\times_{B^2}E^2$ since for $\tau=1$ the path $\gamma_\tau$ is constant. 
\begin{figure}[h]
\begin{center}
\includegraphics[scale=0.4]{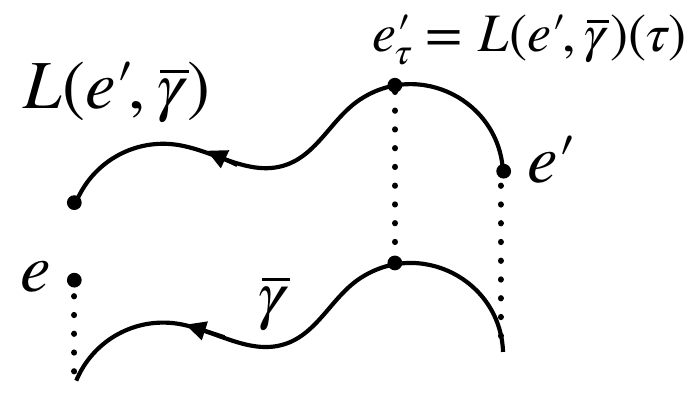}
\caption{The inverse path $\overline\gamma$ and its lift $L(e', \overline\gamma)$}
\label{fig:reverse}
\end{center}
\end{figure}

Besides, if the original triple $(\gamma, e, e')\in B^I\times_{B^2}E^2$ lies in $i(E^2_B)$, then $h_\tau(\gamma, e, e')=(\gamma, e, e')$ for all $\tau$. 
Indeed, in this case $\gamma$ is a constant path and hence  $\gamma_\tau=\gamma$ and $e'_\tau=e'$ for all $\tau\in [0,1]$. 

Now equality (\ref{in:11}) follows from well-known general properties of sectional category. 
\end{proof}
\subsection{}  Next we describe a specific motion algorithm build out of a parametrized motion algorithm and a lifting function. This algorithm may work under variable external conditions assuming that their movement is known in advance. 

For the fibration ({\ref{p}) consider a parametrised motion planning algorithm 
\begin{eqnarray}\label{eq:s}
s: E^2_B\to E^I_B
\end{eqnarray}
and a lifting function
\begin{eqnarray*}
L: E\times_B B^I \to E^I. 
\end{eqnarray*}
For a curve $\gamma\in B^I$ we shall denote by $$P_\gamma: E_{\gamma(0)}\to E_{\gamma(1)}$$ the operator of parallel transport with respect to the lifting function $L$, i.e. 
$P_\gamma(e)=L(e, \gamma)(1)$ for $e\in E$ with $p(e)=\gamma(0)$. {\it We shall assume below that $P_\gamma$ is a homeomorphism for any path $\gamma$; this assumption is automatically satisfied if the lifting function arises from an Ehresmann connection. }

Define the map
\begin{eqnarray}\label{eq:16d}
\tilde s: B^I\times_{B^2}E^2\to E^I
\end{eqnarray}
via 
\begin{eqnarray}\label{eq:17}
\tilde s(\gamma, e, e')(t) = L(s(e, e'_\gamma)(t), \gamma)(t), \quad t\in I. 
\end{eqnarray}
\begin{figure}[h]
\begin{center}
\includegraphics[scale=0.3]{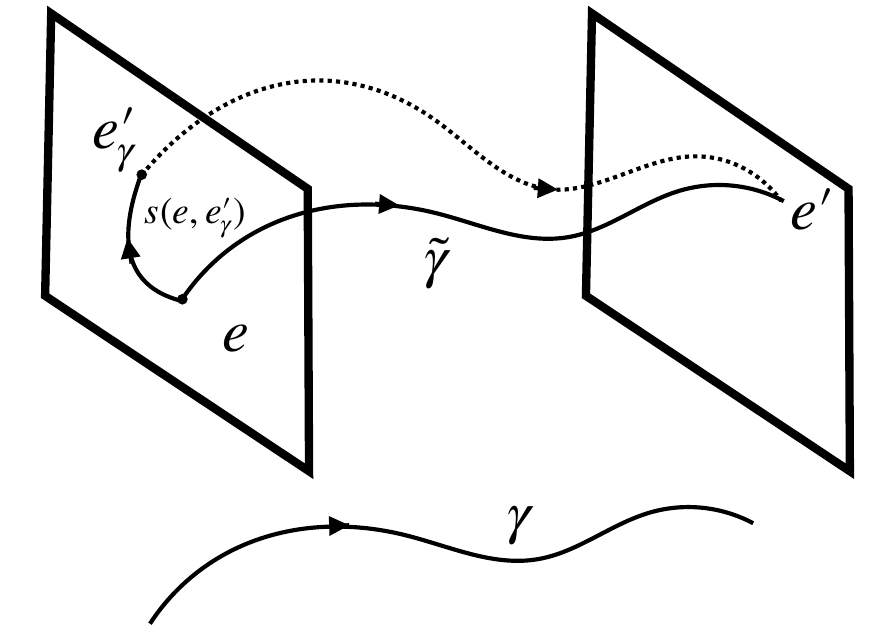}
\caption{The dotted line represents the path $L(e'_\gamma, \gamma)$ and $\tilde\gamma =\tilde s(\gamma, e, e')$.}
\label{fig:lift1}
\end{center}
\end{figure}
Here $\gamma\in B^I$, \, $e, e'\in E$ with $p(e)=\gamma(0)$, $p(e')=\gamma(1)$. The point $e'_\gamma\in E$ is defined by 
$
e'_\gamma = P_\gamma^{-1}(e').
$
This is illustrated by Figure \ref{fig:lift1}. 
We see that the points $e$ and $e'_\gamma$ lie in the same fibre of $p:E\to B$ and the section $s(e, e'_\gamma)\in E^I_B$ is well-defined. 
We have 
$$
p(\tilde s(\gamma, e, e')(t)) = p(L(s(e, e'_\gamma)(t), \gamma)(t))=\gamma(t), 
$$
and
$$\quad \tilde s(\gamma, e, e')(0) =e, \quad \tilde s(\gamma, e, e')(1) = P_\gamma(e'_\gamma)=e'.$$
We conclude that $\tilde s$ is a section of the bundle $\tilde \Pi$, see (\ref{tildea2d}); in other words, $\tilde s$ represents the desired motion algorithm.

\subsection{Complexity of the extended motion algorithm} \label{27} 
\begin{theorem}\label{comp}
The extended algorithm $\tilde s$, see (\ref{eq:16d}), has complexity equal to the complexity of the initial algorithm $s$, see (\ref{eq:s}). 
Therefore, under the assumption that the motion of external conditions is known in advance, the motion planning problem in the situation with changing external conditions has complexity equal the complexity of the initial problem with stationary external conditions.
\end{theorem}
\begin{proof}
Suppose that the initial parametrised algorithm $s$ was continuous on an open subset 
$U\subset E^2_B$. Practically this means that the output of the parametrised algorithm $s$ varies continuously with variations of the input 
(i.e. the initial and final configurations and the state of the external conditions) while the input remains in $U$. 
From formula (\ref{eq:16d}) we see that the extended algorithm $\tilde s$ will be continuous on the subset 
$$
\tilde U= Q^{-1}(U)\subset B^I\times_B E^2_B,
$$
where $Q: B^I\times_B E^2_B\to E^2_B$ is given by 
$$Q(\gamma, e, e')=(e, P_\gamma(e')).$$
Here $e, e'\in E$ and $\gamma\in B^I$ satisfy $p(e)=p(e')=\gamma(0)$. Recall that $P_\gamma$ denotes the operator of parallel transport along $\gamma$. 

Let $k=\tc[p:E\to B]$ denote the complexity of the initial problem. Then there exists an open cover 
$E^2_B =U_0\cup U_1\cup\dots \cup U_k$ with a continuous section $s_i:U_i\to E^I_B$ of fibration (\ref{pi}). 
Then the sets $Q^{-1}(U_i)$ form an open cover of the set 
$B^I\times_B E^2_B$ with continuous sections $\tilde s_i$ of the bundle $\tilde \Pi$. This means that the sectional category of the bundle 
$\tilde \Pi$ is smaller than or equal to $k=\tc[p:E\to B]$. On the other hand, the bundle $\tilde \Pi$ contains the bundle $\Pi$ as a restriction of a subset of the constant paths $B\subset B^I$ which implies the opposite inequality. 
We see that the sectional category of the bundle $\tilde \Pi$ equals 
the parametrised topological complexity $\tc[p:E\to B]$. This gives an alternative proof of Theorem \ref{thm:8}. 
\end{proof}

\section{Infinitesimal lifting functions and infinitesimal reaction mechanisms}\label{sec:ilf}

\subsection{} In the previous sections we considered motion algorithms for autonomous systems which work under the assumptions that either the external conditions are constant during the implementation of the algorithm, or the external conditions change but their variation is known in advance. 
In this section we shall discuss a new class of motion algorithms which are suitable for situations when the above assumptions are not valid. 
These algorithms are applicable when the system has comprehensive sensing capabilities, and performs continuous observations under the changes of the external conditions. 

The main feature of these new motion algorithms is that they are {\it incremental} in nature. 
The resulting motion of an autonomous system can be viewed as the outcome of a continuous sequence of decisions consisting of a choice of the velocity at each moment of time. 
Since the future changes in the external conditions are not known, 
the motion algorithm for the autonomous system can use only information available at the time (which is the current state of the external conditions, the current velocity of the external conditions, and, maybe, the higher derivatives of the motion of the external conditions as well). 
The algorithm produces the instantaneous velocity of the automated system and the automated system moves with this velocity for a very short period of time $t_0$ ({\it the time scale}). Then the new measurements of the external conditions are made and the new instantaneous velocity vector is produced, and the system moves again during $t_0$ seconds, and the process repeats. We formalise this approach below in terms of a differential equation. Note that the time scale may vary and depend on the external conditions and on the velocity of the autonomous system. 

We make an assumption that our autonomous system is equipped with sensing devices and is able to measure the state, the velocity and the acceleration of the external conditions without error and without time delay. 

The infinitesimal motion algorithm, which we define below, 
is  {\it \lq\lq Markov type\rq\rq}\ , i.e. the outcome of the algorithm at time $t$
{\it depends only} on the current state, the velocity and the acceleration of the system and external conditions at time $t$, and not on their past behaviour. 
Clearly, there could be situations when information about the past behaviour of the external conditions can also be useful, however we do not consider this point further in this work.

\subsection{Infinitesimal lifting functions and infinitesimal reaction mechanisms}\label{sec:irm} Consider the fibration of configuration spaces $p: E\to B$ involving the configuration spaces of the automated system and of the external conditions, see \S\ref{sec:13}. In this section we  shall assume that $E$ and $B$ are smooth manifolds and that $p$ is a smooth map. By 
$$\pi_E: TE\to E\quad\mbox{ and}\quad  \pi_B: TB\to B$$ we denote the tangent bundles of $E$ and $B$ correspondingly. The differential $dp: TE\to TB$ is a smooth map satisfying 
\begin{eqnarray}
 \pi_B \circ dp = p\circ \pi_E. 
\end{eqnarray}
We denote by $E\times_B TB$ the space of all pairs $(e, X)\subset E\times TB$ such that 
$p(e) =\pi_B( X)$. One has the map
\begin{eqnarray}\label{eq:frakp}
\mathfrak P: TE\to E\times_B TB, \quad \mbox{where}\quad \mathfrak P(Y)= (\pi_E(Y), dp(Y)), \quad \mbox{for}\quad Y\in TE. 
\end{eqnarray}

\begin{definition}\label{def:ilf}
An infinitesimal lifting function is a continuous map
\begin{eqnarray}\label{irm}
\mathcal L: E\times_B TB\to TE,
\end{eqnarray} 
satisfying $\mathfrak P\circ \mathcal L ={\rm id}$. Moreover, 
in the case when the space $E$ has a nonempty boundary, $\partial E\not=\emptyset$, we shall additionally require that for $e\in \partial E$ 
and 
$X\in T_{p(e)}(B)$ one has 
\begin{eqnarray}\label{eq:21d}
\mathcal L(e, X)\in T^{+}_e\subset T_e(E).
\end{eqnarray}
Here $T^{+}_e\subset T_e(E)$ denotes the closed half-space of the tangent space $T_e(E)$ consisting of vectors pointing inwards. 
\end{definition}
Note that the tangent space $T_e(E)$ is divided by the subspace $T_e(\partial E)\subset T_e(E)$ of the vectors tangent to the boundary into two half-spaces
$
T_eE = T_e^+\cup T_e^-, $
where $T_e^+ \cap T_e^- =T_e(\partial E).$
The half-space $T_e^+$ consists of {\it the admissible velocity vectors}, i.e. the vectors which point either inside $E$ or are tangent to the boundary $\partial E$.

\subsection{} Given a point $e\in E$ and a tangent vector $X\in T_{p(e)}B$, the function $\mathcal L$ produces a tangent vector $Y=\mathcal L(e, X)$ in the tangent space $T_eE$ satisfying $dp_e(Y)=X$. 

We shall now restate the purpose of $\mathcal L$ using the language of robotics: {\it given a state $e\in E$ of the system and of the external conditions which are consistent with each other, and the velocity of the motion of the external conditions $X$, the function $\mathcal L$ produces the velocity $Y=\mathcal L(e, X)$ of the system consistent with $X$.} The condition (\ref{eq:21d}) guarantees that the suggested velocity vector is indeed achievable.
This explains why an infinitesimal lifting function serves as {\it an infinitesimal reaction mechanism}. 

\subsection{}\label{sec:44} For a smooth curve $\gamma: J\to B$, where $J=[a, b)$ with $a<b\le \infty$ and for a point $e\in E$ satisfying $p(e)=\gamma(a)$, the lift $\tilde \gamma: J\to E$, where $\tilde \gamma=L(e, \gamma)$, is the solution of the initial value problem 
\begin{eqnarray}\label{eq:21a}
\frac{d}{dt}\left[\tilde\gamma(t)\right] = \mathcal L\left(\tilde\gamma(t), \frac{d}{dt}\left[\gamma(t)\right]\right), \quad \tilde \gamma(a)=e. 
\end{eqnarray}
If $\mathcal L$ is smooth the solution to (\ref{eq:21a}) exists locally and is unique. 
An infinitesimal lifting function $\mathcal L$  is said to be {\it complete} if the solution to (\ref{eq:21a}) exists globally for any pair 
$(e, \gamma)$ where $\gamma: J\to B$ is a smooth curve and  $p(e)=\gamma(a)$. 

A complete infinitesimal lifting function $\mathcal L: E\times_B TB\to TE$ defines {\it a smooth lifting function} 
$$
L: E\times_BC^{\infty}(J,B)\to C^{\infty}(J, E),
$$ 
where the symbols $C^{\infty}(J,B)$ and $C^{\infty}(J, E)$ denote the spaces of smooth maps $J\to B$ and $J\to E$ correspondingly. 
For $e\in E$ and $\gamma\in C^{\infty}(J,B)$ satisfying $p(e)=\gamma(a)$ the image $L(e, \gamma)=\tilde \gamma$ is the curve $\tilde\gamma: J\to E$ which is the solution of (\ref{eq:21a}). 

%in the sense of Definition \ref{liftingf} as follows. 

%Any infinitesimal lifting function $\mathcal L$ is complete if the fibre of the bundle $p:E\to B$ is compact. \marginpar{reference}

\subsection{Ehresmann connections.}\label{sec:ehresmann} The classical construction of Ehresmann connections  \cite{E}, \cite{W} leads to infinitesimal lifting functions which have the additional property of {\it linearity}. We recall this notion below. 

For a smooth bundle $p: E\to B$, 
the family of the subspaces $\mathcal V=\{V_e\}_{e\in E}$, where
$$V_e=\ker[dp_e: T_eE\to  T_{p(e)}B]$$
form {\it the vertical distribution}. This is the space of tangent vectors which are tangent to the fibres. The family of vertical tangent spaces $V_e$ 
form a sub-bundle 
\begin{eqnarray}\label{eq:vert}
T^VE\to E, \quad T^VE\subset TE
\end{eqnarray}
of the tangent bundle $TE$.
{\it An Ehresmann connection} for the bundle $p:E\to B$ is a differentiable distribution 
\begin{eqnarray}\label{ehr}
\mathcal H=\{H_e\}_{e\in E}\end{eqnarray}
which is complementary to the vertical distribution $\mathcal V$.
Given an Ehresmann connection (\ref{ehr}), each tangent space $T_eE$, where $e\in E$, is the direct sum $$T_eE=V_e\oplus H_e$$ and the differential 
$dp_e$ maps the horizontal subspace $H_e$ {\it isomorphically} onto $T_{b}B$, where $b=p(e)\in B$. Thus, for every tangent vector $X\in T_bB$ and for every $e\in E$ with $p(e)=b$ there exists {\it a unique lift} 
$$Y \, \in H_e\subset T_eE \quad\mbox{with}\quad  dp_e(Y)=X.$$
We shall denote 
$$
Y = \mathcal L(e, X).
$$
This defines  {\it an infinitesimal lifting function} associated with the Ehresmann connection $\mathcal H$, 
\begin{eqnarray}\label{ilfe}
\mathcal L: E\times_B TB\to TE.
\end{eqnarray} 

 \begin{definition}
 An Ehresmann connection is said to be complete if the associated infinitesimal lifting function $\mathcal L$ is complete, as defined in \S \ref{sec:44} above, i.e. if the ordinary differential equation (\ref{eq:21a}) admits a global solution for any initial condition. 
 \end{definition}
% For a every smooth regular curve $\gamma: [a,b)\to B$ in the base $B$ admits a lift 
% $\tilde \gamma $ with an arbitrary initial condition $\tilde \gamma(a)=e$, where $p(e)=\gamma(a)$; compare .
% 

We cite the following result:
 
\begin{theorem}\cite{Hoyo} 
Every smooth locally trivial bundle $p:E\to B$ admits a complete Ehresmann connection. 
\end{theorem}

Next we state the following obvious lemma:
\begin{lemma}\label{lm:ehr}
The infinitesimal lifting function (\ref{ilfe}) arising from an Ehresmann connection has the following linearity properties:
\begin{eqnarray}\label{prop1}
\mathcal L(e, X_1+X_2)= \mathcal L(e, X_1)+\mathcal L(e, X_2) \quad \mbox{for} \quad X_1, X_2\in TB,
\end{eqnarray}
and 
\begin{eqnarray}\label{prop2}
\mathcal L(e, \lambda X)=\lambda \mathcal L(e, X)\quad \mbox{for every}\quad \lambda\in \R.
\end{eqnarray}
And vice versa, any infinitesimal lifting function (as defined in Definition \ref{def:ilf}) possessing the linearity properties (\ref{prop1}) and (\ref{prop2}) originates from an Ehresmann connection. 
\end{lemma}
We refer to the book  \cite{KMS} for additional information and some general results.

\subsection{Affine structure on the space of infinitesimal lifting functions and reaction forms} 
Note that the space of all infinitesimal lifting functions (infinitesimal reaction mechanisms) has a natural affine structure. 
Let 
$$\mathcal L_1, \mathcal L_2: E\times_B TB\to TE$$ be two infinitesimal lifting functions, and 
let 
$
\theta: E\to \R
$
be a smooth function. We can now define $\mathcal L_\theta: E\times_B TB\to TE$ as an affine combination
\begin{eqnarray}\label{eq:aff}
\mathcal L_\theta (e,X)  = \theta (e) \cdot \mathcal L_1(e, X)+ (1-\theta(e))\cdot \mathcal L_2(e, X).
\end{eqnarray}
If $Y_i=\mathcal L_i(e, X)$ for $i=1, 2$, then $\mathcal L_\theta(e, X)=\theta(e)\cdot Y_1+(1-\theta(e))\cdot Y_2$ and 
$$
dp_e(\mathcal L_\theta (e,X)) = \theta(e)\cdot dp_e(Y_1)+(1-\theta(e))\cdot dp_e(Y_2)= \theta(e)\cdot X + (1-\theta(e))\cdot X=X. 
$$
We see that $\mathcal L_\theta$ is an infinitesimal lifting function.

One may be interested in the linear space associated with the space of infinitesimal lifting functions. This is the space of {\it reaction forms} consisting of maps
\begin{eqnarray}
\ell: E\times_BTB \to T^{V}E
\end{eqnarray}
such that the composition
\begin{eqnarray}
E\times_B TB\, \, \stackrel{\ell} \to T^VE\, \,  \stackrel{\pi_E}\to\,  E
\end{eqnarray}
coincides with the projection $ E\times_B TB\to E$ on the first factor. The latter condition simply means that $\ell(e, X)$ lies in the space $T^V_eE$, i.e. $\ell(e, X)$ is a vector tangent to the fibre at  the point $e$. 
The reaction form $\ell$ produces a vertical tangent vector representing the velocity of the autonomous object 
$$\ell(e, X)\in T^VE, \quad \mbox{where}\quad e\in E, \quad X\in T_{p(e)}B,$$ 
which depends on the velocity of the external conditions $X\in T_{p(e)}B$; the vector  $\ell(e, X)$ can be viewed as {\it \lq\lq a reaction\rq\rq}\ to the change represented by $X$.

Note that the reaction form $\ell(e, X)$ does not need to be linear with respect to $X$, as illustrated by Example \ref{ex:46} and some other examples below.

\subsection{Combining reaction forms}
The set of all reaction forms is a linear space: 
if $\ell_1$ and $\ell_2$ are two reaction forms on the bundle $p:E\to B$ then their sum $\ell_1+\ell_2$ is a well defined reaction form. Moreover, if 
$\mathcal L: E\times_BTB\to TE$ is an infinitesimal reaction mechanism 
and $\ell:  E\times_BTB \to T^VE$ is a reaction form 
then the sum 
$$\mathcal L' = \mathcal L+\ell: E\times_BTB\to TE$$ 
is another infinitesimal reaction mechanism. 

 The above discussion leads to the following:
\begin{corollary}\label{cor:45}
A general infinitesimal reaction mechanism $\mathcal L: E\times_BTB\to TE$ can be constructed as the sum
\begin{eqnarray}
\mathcal L= \mathcal L_0+ \ell_1+\ell_2+\dots+\ell_N,
\end{eqnarray}
where $\mathcal L_0$ is a fixed infinitesimal reaction mechanism (say, an Ehresmann connection) and $\ell_1, \dots, \ell_N$ are reaction forms. 
\end{corollary}
Typically, the support of the reaction forms $\ell_1, \dots, \ell_N$ is concentrated near the boundary components of the compound configuration space $E$.

\begin{example}[The \lq\lq Moving away manoeuvre\rq\rq]\label{ex:46}{\rm 
Here we give an example of a nonlinear reaction form. Suppose that the fibre $X$ of the bundle of configuration spaces $p:E\to B$ is a manifold with boundary, $\partial X\not=\emptyset$. Then $\partial E\not=\emptyset$. 
Let $N$ be a smooth vector field on $E$ which is tangent to the fibres of the fibration $p: E\to B$, i.e. such that 
$$dp_\ast(N)=0.$$ 
We shall assume that the vector field $N$ has support in a neighbourhood of the boundary $\partial E$
and  that $N$ points inside of $E$ at the boundary $\partial E$. 
\begin{figure}[h]
\begin{center}
\includegraphics[scale=0.4]{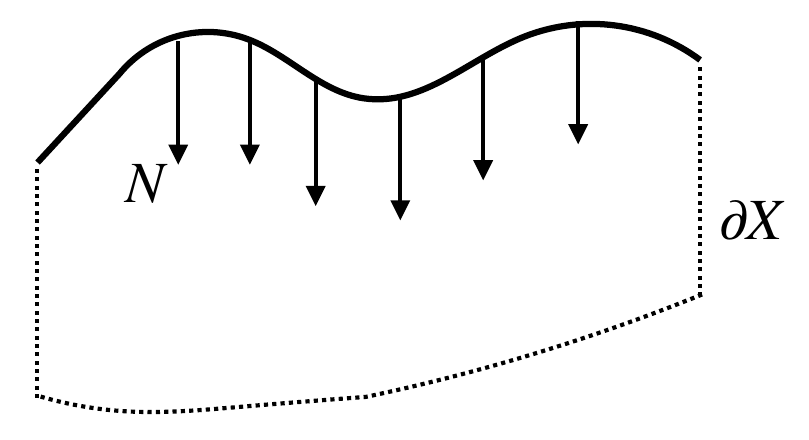}
\caption{Vector field $N$ near $\partial X$.}
\label{default}
\end{center}
\end{figure}

We may define a reaction form $\ell$ by
\begin{eqnarray}\label{eq:elln}
\ell(e, X) = \psi(e)\cdot |X|\cdot N(e),
\end{eqnarray}
where $|X|$ denotes the length of the tangent vector $X\in T_{p(e)}B$ with respect to a Riemannian metric and 
$$\psi: E\to [0,1]$$ is a smooth function 
which vanishes outside a neighbourhood of the boundary $\partial E$ and equals $1$ near $\partial E$. 
We see that (\ref{eq:elln}) is a positive homogeneous (but not linear) reaction form. The effect of $\ell$ can be described as {\it \lq\lq pushing away from the boundary\rq\rq}.  

We say that $\psi$ is {\it an actuation function}; it switches on the action of the reaction form only when the system approaches the boundary $\partial E$. 

Note that the boundary $\partial E$ limits the space of admissible configurations and the system should {\it start \lq\lq reacting\rq\rq}\ when it is near the boundary.  
The configurations lying on the boundary $\partial E$ represent tangencies with the obstacles and the velocity of our system at such configurations is limited. 
}
\end{example}

\begin{lemma}\label{lm:47}
Let $\mathcal L: E\times_BTB\to TE$ be a linear infinitesimal reaction mechanism, i.e. an Ehresmann connection. If $e\in \partial E$ then for any 
tangent vector 
$X\in T_{p(e)}B$ the tangent vector $Y=\mathcal L(e, X)$ belongs to the codimension one subspace $T_e(\partial E)\subset T_eE$ of vectors tangent to the boundary. 
\end{lemma}
\begin{proof} As in Definition \ref{def:ilf}, the tangent space $T_eE$ is divided by the subspace $T_e(\partial E)$ into two half-spaces
$$
T_eE = T_e^+\cup T_e^-, \quad \mbox{where}\quad T_e^+ \cap T_e^- =T_e(\partial E).
$$
Here, $T_e^+$ is the space of admissible velocity vectors, i.e. the vectors which point either inside $E$ or are tangent to the boundary $\partial E$. 
According to Definition \ref{def:ilf}, an infinitesimal reaction mechanism should satisfy 
$$\mathcal L(e, X)\in T_e^+\quad\mbox{ for any}\quad  X\in T_{p(e)}B$$
since only the vectors from $T_e^+$ are velocity vectors of admissible motions of the system. 
If $Y=\mathcal L(e, X)\in T_e^+$ then, due to linearity, one has 
$\mathcal L(e, -X)=-Y\in T_e^+$ and hence 
$Y\in T_e^- \cap T_e^+=T_e(\partial E).$
Thus, we see that $\mathcal L(e, X)\in T_e(\partial E)$ for all $e\in \partial E$ and $X\in T_{p(e)}B$. 
\end{proof}
\begin{corollary}\label{cor:48}
No linear infinitesimal reaction mechanism (an Ehresmann connection) can implement the task of {\it pushing away from the boundary} and hence, the nonlinearity which appears in 
(\ref{eq:elln}), is unavoidable. 
\end{corollary}
\subsection{Navigation target} \label{sec:target} Our ultimate goal is to create an infinitesimal reaction mechanism for the autonomous system which 
eventually takes our system to a prescribed target state, regardless of the behaviour of the external conditions. 

The target is a specific state of our system which we wish it to achieve. Typically, the target state can be consistent with various external conditions. If a point 
$b\in B$ represents a state of the external conditions then the target consistent with $b$ is a specific point $e\in p^{-1}(b)$. 
This suggests that the target can be viewed as being a section 
of the fibration $p:E\to B$, i.e. a continuous map, 
$s: B\to E\quad\mbox{satisfying}\quad  p\circ s=1_B,$ 
where $1_B: B\to B$ is the identity map. 

As we shall see in the following  \S \ref{sec:5}, 
in general {\it the target} should be viewed as a section $s: A\to E$ defined on {\it a subset $A\subset B$ of admissible states} of external conditions, see \S \ref{sec:53}. For the 
external conditions $b\in B-A$ the target {\it is not available}.

%In the following section we shall analyse a specific situation of this kind and will see limitations of this approach. In this subsection we formalise the notion of a target in our setting involving fibrations of configuration spaces $p: E\to B$, see (\ref{p}). 
%

\section{Example: two autonomous robots (vehicles) on the plane} \label{sec:5} 
The purpose of this section is to illustrate the material of the previous \S \ref{sec:ilf} by a specific toy example which resembles the situation in automated driving. 
We shall examine several infinitesimal reaction mechanisms and their effect on the motion of the system. 

\subsection{} We consider in this section two robots (vehicles), $M$ and $N$,  moving along the plane $\R^2$. The vehicle $M$ represents an automated system (ego) and our task is to design its motion algorithm. $N$ is another vehicle which moves according to its own goals and $M$ \lq\lq does not know\rq\rq\ the future behaviour of $N$. However, we assume that $M$ is able to measure the positions of both systems $M$ and $N$ and their velocities.

We shall represent $M$ and $N$ by Euclidean disks of radius $1$, and the symbols $c_M\in \R^2$ and $c_N\in \R^2$ will denote their centres. 
For us the disks are not exactly the robots themselves; they rather represent certain {\it safety zones} around each of the vehicles. 
\begin{figure}[h]
\begin{center}
\includegraphics[scale=0.4]{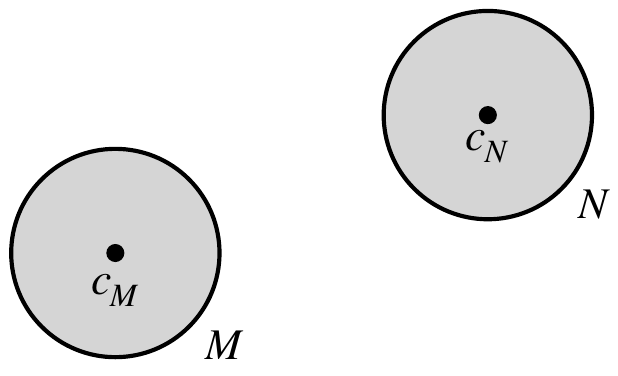}
\caption{A configuration of two robots $M$ and $N$.}
\label{fig:robots}
\end{center}
\end{figure}
The safety requirements impose the rule that {\it the interiors of the disks must be disjoint}.

We think of the robots as rigid bodies; their states are determined by the positions of their centres $c_M, c_N\in \R^2$,
 %and by the angles $\phi_M, \phi_N\in S^1$, indicating their rotational positions, 
 see Figure \ref{fig:robots}. 
%However, because of angular symmetry, we shall ignore the angles $\phi_M, \phi_N$ and consider the positions of the centres $c_M, c_N$ as their states. 
The admissible configurations of $M$ and $N$ satisfy
\begin{eqnarray}\label{eq:32a}
|c_M-c_N|\ge 2.
\end{eqnarray}
In other words, we allow the states when the safety disks are disjoint or they
touch each other along their boundaries and we require that the interiors of the disks are disjoint at all times. 

As we focus on a motion algorithm for $M$, we shall view the robot $N$ as an obstacle (the external conditions). The motion algorithm for $M$ must (a) ensure that no collision with $N$ occurs and (b) $M$ arrives at a prescribed target state at certain moment of time. With this purpose in mind we shall analyse infinitesimal reaction mechanisms for $M$ as defined in \S \ref{sec:irm}. 

\subsection{} We denote by $E$ the configuration space of the pair $M, N$. Clearly, $E$ can be identified with 
\begin{eqnarray}\label{eq:e}
E= \{(c_M, c_N)\in \R^2\times\R^2\, ;\,  |c_M-c_N|\ge 2\}.
\end{eqnarray}
We note that $E$ is a manifold with boundary. The configuration space of $N$ (the external conditions) is 
$
B=\{c_N\} = \, \R^2
$
and the fibration of configuration spaces 
$
p:E\to B
$  
 (see \S \ref{sec:13}) is the map 
$(c_M, c_N)\mapsto c_N.$
The fibre $p^{-1}(c_N)$ can be identified with $\R^2-D$, where $D\subset \R^2$ is the open ball of radius $2$ with centre 
$c_N$. It is easy to see that the fibration $p:E\to B$ is trivial in this case, i.e. it admits a global trivialisation.

\subsection{}\label{sec:53t} {\it The target state} is a prescribed state $c_M^{\rm tag}\in \R^2$ to which the robot $M$ aims to arrive. 
As in \S \ref{sec:target} we may view the target as a section 
$s: A\to E$ given by 
$s(c_N)=c_M^{\rm tag}$. Here $A\subset B$ is the set 
$$A=\{c_N\in \R^2; \, |c_N-c_M^{\rm tag}|\ge 2\}.$$ 
The set $A$ is the set of external conditions under which the target state is available. Note that if $|c_N-c_M^{\rm tag}|< 2$ then the disc $N$ blocks the target state $c_M^{\rm tag}$ 
and the system $M$ is not able to get to the target. 

\subsection{} Consider the differential $dp: TE\to TB$. A tangent vector at a configuration $e=(c_M, c_N)
\in E$ can be identified with the tuple 
$$(v_M, v_N)\in \R^2\times\R^2,$$ where $v_M, v_N\in \R^2$ are the velocities of the centres $c_M, c_N$.
If $|c_M-c_N|>2$ then any such tangent vector can be realised by a smooth motion in the configuration space $E$. However if $|c_M-c_N|=2$ then we need to require that 
\begin{eqnarray}\label{eq:26}
\langle v_M-v_N, c_M-c_N\rangle \, \ge 0.
\end{eqnarray}
This condition guarantees that the distance between the centres of the disks does not decrease infinitesimally and hence the system remains in the space of admissible configurations, as  given by (\ref{eq:32a}).

An infinitesimal reaction mechanism $\mathcal L: E\times_B TB\to TE$ (as defined in Definition \ref{def:ilf}) in this specific case acts as follows. 
If $X=v_N$ and $e=(c_M, c_N)$ then 
\begin{eqnarray}\label{tang}
Y=\mathcal L(e, X) = (v_M, v_N),
\end{eqnarray}
i.e. $\mathcal L$ \lq\lq supplies\rq\rq\ the velocity $v_M\in \R^2$ such that the tangent vector (\ref{tang}) is admissible. 
If $|c_M-c_N|>2$ then $v_M$ can be arbitrary. 
However, if $|c_M-c_N|=2$ the velocity vector $v_M$ must satisfy inequality 
(\ref{eq:26}). 
We mention below a few specific infinitesimal reaction mechanisms available for the system $M$ and their effects on the mutual position of $M$ and $N$. 

\subsection{Infinitesimal reaction mechanism 1} \label{sec:53} Setting 
\begin{eqnarray}\label{rm1}
v_M=v_N 
\end{eqnarray} 
defines an infinitesimal reaction mechanism. 
In this case the system $M$ simply \lq\lq copies\rq\rq\ the velocity of $N$. The rule
(\ref{rm1}) defines an Ehresmann connection on the bundle $p:E\to B$. If the autonomous system $M$ follows the reaction mechanism (\ref{rm1}) then 
$
(c_M-c_N)'=v_M-v_N=0,
$
i.e. the vector $c_M-c_N$ remains constant, and the system $M$ essentially repeats the movements of $N$ shifted by a fixed nonzero vector, see Figure
\ref{fig:rep}.
\begin{figure}[h]
\begin{center}
\includegraphics[scale=0.25]{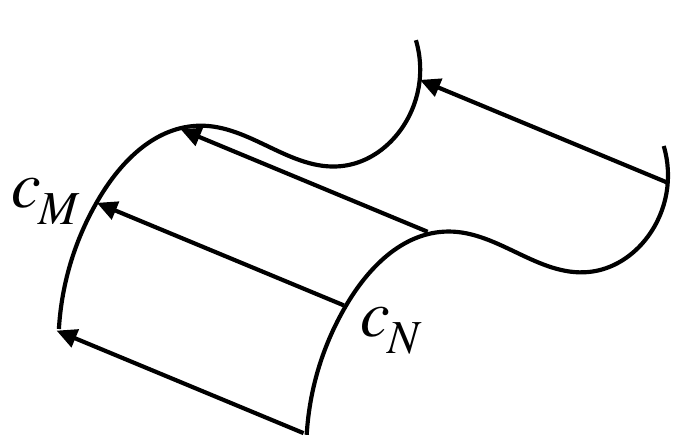}
\caption{Motion of $c_M$ and $c_N$ under the infinitesimal reaction mechanism 1.}
\label{fig:rep}
\end{center}
\end{figure}

The systems $M$ and $N$ will never collide in this case. 
If initially the balls $M$ and $N$ were touching, i.e. if $|c_M(0)-c_N(0)|=2$, then $|c_M(t)-c_N(t)|=2$ for any $t\ge 0$, i.e. the balls will remain touching all the time. This is consistent with Lemma \ref{lm:47}. 

\subsection{Infinitesimal reaction mechanism 2}\label{sec:noreaction} {\it \lq\lq No reaction\rq\rq}\ corresponds to setting 
$v_M=0$. This will define an infinitesimal reaction mechanism in areas where $|c_M-c_N|>2$, i.e. when the balls are not touching. 
In the case when $|c_M-c_N|=2$ this reaction rule may not satisfy (\ref{eq:26}). To rectify this problem we set 
\begin{eqnarray}\label{eq:37b}
v_M=\psi(|c_M-c_N|)\cdot v_N,
\end{eqnarray}
where $\psi: [0, \infty)\to [0,1]$ is an actuation function, i.e. a smooth function which is identically 1 on $[0,2]$ and has support in $[0,3]$, see Figure \ref{fig:gxa}.
\begin{figure}[h]
\begin{center}
\includegraphics[scale=0.4]{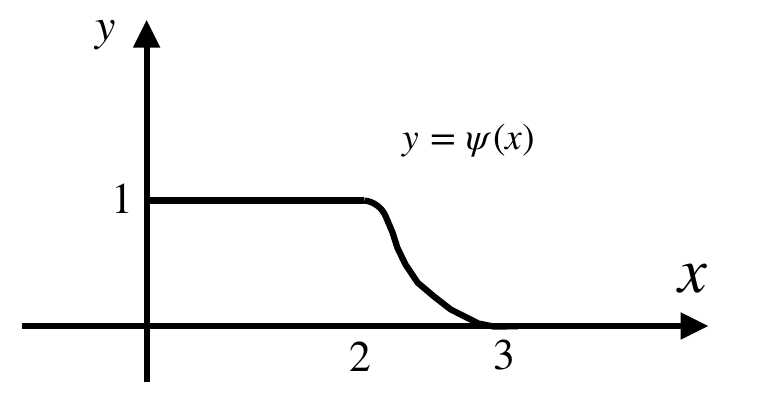}
\caption{The graph of the actuation function $\psi$.}
\label{fig:gxa}
\end{center}
\end{figure}

Of course, the upper bound $3$ can be selected differently; the actual choice of the upper bound is important and in practical situations it may depend on the technical abilities of the systems $M$ and $N$.
 
Formula (\ref{eq:37b}) defines an Ehresmann connection. 
In the area $|c_M-c_N|\ge 3$ the infinitesimal reaction mechanism (\ref{eq:37b}) acts as \lq\lq no reaction\rq\rq\ and if $|c_M-c_N|= 2$ (the balls are touching) 
we shall have $v_M=v_N$, i.e. the response will be identical to the previous example.

\subsection{Infinitesimal reaction mechanism 3} \label{sec:54}
As another example of a reaction mechanism consider the following rule:
\begin{eqnarray}\label{eq:30}
v_M=v_N+\lambda\cdot (c_M-c_N), \quad \mbox{where}\quad \lambda\in \R\quad \mbox{is a constant}.
\end{eqnarray}
The case $\lambda=0$ reduces to the example of \S \ref{sec:53}. We shall assume below that $\lambda\not=0$. 
The infinitesimal reaction mechanism (\ref{eq:30}) is nonlinear and hence is not an Ehresmann connection.
\begin{lemma}
If the autonomous system $M$ employs the infinitesimal reaction mechanism (\ref{eq:30}) then, regardless of the behaviour of $N$,
the vector $c_M-c_N$ has constant direction and 
the distance $d(t)=|c_M(t)-c_N(t)|$ between the centres of the disks $M$ and $N$ satisfies 
\begin{eqnarray}\label{eq:31}
d(t)=e^{\lambda t}\cdot d(0), \quad t\ge 0.
\end{eqnarray}
\end{lemma}
\begin{proof} Differentiating the function $d^2(t)=\langle c_M-c_N, c_M-c_N\rangle$ with respect to the time $t$ and using (\ref{eq:30}) one obtains 
$$
\frac{d}{dt}(d^2)=2 \lambda d^2,
$$
which implies (\ref{eq:31}) by integration. On the other hand, for any fixed vector $v\in \R^2$ the scalar product $\langle v, c_M-c_N\rangle$ satisfies 
$\langle v, c_M-c_N\rangle'=\lambda \langle v, c_M-c_N\rangle$, i.e. $\langle v, c_M-c_N\rangle=Ce^{\lambda t}$, where $C$ is a constant. Hence, if $v$ is perpendicular to the vector 
$c_M-c_N$ at time $t=0$ then $C=0$ and therefore $v\perp (c_M-c_N)$ for all $t$, i.e. $c_M$ lies on the line passing through $c_N$ in the direction perpendicular to the vector $v$. 
\end{proof}

In the case $\lambda>0$, the automated vehicle $M$ moves away from $N$ exponentially far. In the case $\lambda<0$ the distance between $M$ and $N$ exponentially decreases and at a certain moment of time it will become $d=2$, i.e. we shall arrive at the boundary of the configuration space $E$; at this moment the vehicles $M$ and $N$ will touch each other. 

We see that the infinitesimal motion algorithms (\ref{eq:30}) 
with $\lambda>0$ {\it \lq\lq guarantee safety,\rq\rq}\  i.e. under the assumption $\lambda>0$
the vehicle $M$ will achieve safety regardless of the behaviour of $N$. 

\subsection{Infinitesimal reaction mechanism 4} \label{sec:55}
As another example, consider the infinitesimal reaction mechanism of the form
\begin{eqnarray}\label{eq:32}
v_M=v_N+\mu \cdot(c_M-c_N)^{\perp},
\end{eqnarray}
where $(c_M-c_N)^{\perp}$ is the vector obtained from $c_M-c_N$ by a 90 degrees rotation in the positive direction, and $\mu\in \R$. 
If $c$ denotes the vector $c_M-c_N$ then (\ref{eq:32}) gives $\dot c=\mu c^\perp$ and we have 
$$
\langle c, c\rangle' = 2 \cdot \langle \dot c, c\rangle =0,
$$
i.e. ~the norm $|c|=r$ is constant. We may write $c=r e^{i\phi}$ and then $\dot c=i\phi'\cdot re^{i\phi}$ and since $c^\perp=ic= ire^{i\phi}$ we obtain
$
\phi'=\mu.
$
We see that in this case the distance between the centres $c_M$ and $c_N$ remains constant and $M$ rotates around $N$ with angular velocity 
equal the value of 
the parameter $\mu$. In the case $\mu=0$ we again obtain the solution of \S \ref{sec:53}.

\section{Linear Infinitesimal Reaction Mechanisms with constant coefficients} \label{sec:6}

\subsection{} In this section we continue to consider the model with two balls on the plane as described in \S \ref{sec:5} and analyse several linear infinitesimal reaction mechanisms for this model. 
Recall that in this model we have two round balls of radius 1 on the plane, where $M$ is our ego vehicle and $N$ is \lq\lq the other vehicle\rq\rq\ whose motion is not known to $M$ in advance. 

The differential equations (\ref{eq:21a}) can easily be integrated in the case of linear infinitesimal reaction mechanism {\it with constant coefficients} which we consider in this section. 
The general infinitesimal reaction mechanism of this kind has the form 
\begin{eqnarray}\label{eq:irm1}
v_M=\alpha\cdot v_N + \beta\cdot v_N^{\perp},\quad \alpha, \beta\in \R,
\end{eqnarray}
where $\alpha$ and $\beta$ are constants.  Here $v_N$ is the velocity of the ball $N$, i.e. $v_N=\dot c_N$, and $v_N^\perp$ stands for the result of 90 degrees rotation of the vector $v_N$ in the positive direction. We shall assume that  $(\alpha, \beta)\not=(0,0)$, i.e., we exclude the possibility that both parameters $\alpha$ and $\beta$ simultaneously vanish, compare \S \ref{sec:noreaction}. 

If $c_N(t)$ and  $c_M(t)$ denote the motions of the centres of the balls $N$ and $M$ then equation (\ref{eq:irm1}) means
\begin{equation}\label{eq:irm2}
\dot c_M = \alpha\cdot \dot c_N + \beta\cdot J \dot c_N = (\alpha I+\beta J)\dot c_N =A \dot c_N,
\end{equation}
where $I=\left[\begin{array}{cc}
1 & 0\\
0 & 1
\end{array}\right]
$ is the identity map $I:\R\to \R$ and $J:\R\to \R$ is the 90 degrees rotation, i.e. $J=\left[
\begin{array}{cc}
0 & -1\\
1 & 0
\end{array}
\right]$, and $A=\alpha I+\beta J$, i.e. $A = \left[ \begin{array}{cc} \alpha& -\beta\\ \beta&\alpha\end{array}\right]$. 
From (\ref{eq:irm2}), integrating, we obtain 
\begin{equation}
c_M(t) = Ac_N(t) - c_0,
\end{equation}
where $c_0\in \R^2$ is a fixed vector.

\subsection{} There are two essentially distinct cases.
If $A=I$, i.e., $\alpha =1$ and $\beta=0$, then $$c_M(t)=c_N(t)-c_0,$$ where $c_0$ is a constant vector. This is the infinitesimal reaction mechanism of 
\S \ref{sec:53}. In this case the system $M$ 
{\it \lq\lq repeats\rq\rq}\  
the motion of $N$ with a shift by vector $c_0$. The distance between the systems $M$ and $N$ remains constant and is equal $|c_0|$, i.e. 
$$|c_M(t)-c_N(t)|=|c_0| \quad \mbox{for all}\quad t\ge 0.$$
\subsection{} Consider now the remaining general case $(\alpha, \beta)\not= (1, 0)$.  We have 
\begin{eqnarray}\label{eq:mn}
c_M(t)-c_N(t) = Bc_N(t) -c_0, \quad t\in \R,
\end{eqnarray}
where $$B=A-I= \left[ \begin{array}{cc} \alpha-1& -\beta\\ \beta&\alpha-1\end{array}\right].$$
Since $\det B = (\alpha-1)^2+\beta^2\not=0$, we see that the matrix $B$ is invertible. We may write $c_0=B\tilde c_0$ for a unique vector $\tilde c_0\in \R^2$, and then 
equation (\ref{eq:mn}) can be presented as 
\begin{equation}\label{eq:mn2}
c_M(t)-c_N(t) = B(c_N(t) -\tilde c_0).
\end{equation}
We see from (\ref{eq:mn2}) that the distance between the balls $M$ and $N$ tends to 0 when $c_N(t)$ tends to $\tilde c_0$. 
The vector $\tilde c_0\in \R^2$ can be expressed through the positions of the vehicles $M$ and $N$ at time $t=0$, i.e. via $c_M(0)$ and $c_N(0)$, as follows:
\begin{equation}\label{eq:c0}
\tilde c_0= (I+B^{-1})\cdot c_N(0) - B^{-1}\cdot c_M(0).
\end{equation}

Below the notation $|B|$ stands for the norm of the operator $B:\R^2\to \R^2$. 

\begin{lemma}\label{lm:61} Assume that the autonomous (ego) vehicle $M$ employs the linear infinitesimal reaction mechanism (\ref{eq:irm1}) with constant coefficients, and $(\alpha, \beta)\not= (1, 0)$. Consider the domain 
\begin{eqnarray}
\mathcal D=\{B^{-1}(x)+\tilde c_0; |x|\le 2\}\subset \R^2,
\end{eqnarray}
the image of the ball of radius $2$ with centre at the origin under the operator $B^{-1}$ shifted by the vector $\tilde c_0$, given by (\ref{eq:c0}).
Consider also the ball $D\subset \R^2$ with centre $\tilde c_0$ and radius $2\cdot |B|^{-1}$, as well as the ball $D'\subset \R^2$ with centre $\tilde c_0$ and radius 
$2\cdot |B^{-1}|$. 
We claim:
\begin{enumerate}
\item[{(a)}] The vehicles $M$ and $N$ are in collision at time $t>0$ if and only if the centre $c_N(t)$ of the vehicle $N$ lies in the 
domain $\mathcal D$, i.e. if $c_N(t)\in \mathcal D$;  

\item[{(b)}] One has $D\subset \mathcal D\subset D'$ and hence $c_N(t)\in D$ implies a collision  at time $t$ and $c_N\notin D'$ implies the absence of a collision at time $t$.
\end{enumerate}
\end{lemma}
\begin{figure}[htbp]
\begin{center}
\includegraphics[scale=0.5]{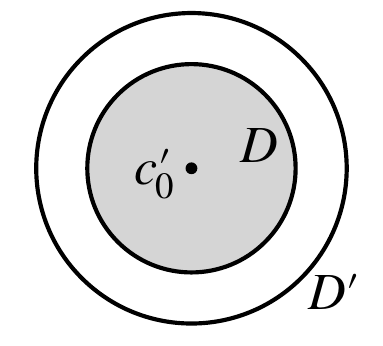}
\caption{The balls $D\subset D'$ with the centre $\tilde c_0$.}
\label{fig:1a}
\end{center}
\end{figure}
\begin{proof} Equation (\ref{eq:mn2}) gives 
$
c_N(t) = B^{-1}(c_M(t)-c_N(t))+\tilde c_0
$
and (a) follows since a collision at time $t$ happens iff $|c_M(t)-c_N(t)|\le 2$. 

To prove (b) we note that $y\in \mathcal D$ implies 
$y-\tilde c_0=B^{-1}(x)$ where $|x|\le 2$. Then $|y-\tilde c_0|=|B^{-1}x|\le |B^{-1}|\cdot |x|\le 2\cdot |B^{-1}|$. Thus, we see that $\mathcal D\subset D'$. 
If $y\in D$, i.e. $|y-\tilde c_0|\le 2\cdot |B|^{-1}$, then $y-\tilde c_0=B^{-1}(B(y-\tilde c_0))$ and $|B(y-\tilde c_0)|\le |B|\cdot |y-\tilde c_0|\le 2$, i.e. $y\in \mathcal D$. 
\end{proof}
%\vskip 1cm
%we see that if
%\begin{equation}\label{eq:6}
%|c_N(t)-\tilde c_0|< 2\cdot |B|^{-1},
%\end{equation}
%then 
%$$
%|c_M(t)-c_N(t)| \le |B|\cdot |c_N(t)-\tilde c_0|<2,
%$$
%i.e. the balls $M$ and $N$ are in collision; this proves (a). 
%On the other hand, if 
%\begin{equation}\label{eq:7}
%|c_N(t)-\tilde c_0|\ge 2\cdot |B^{-1}|,
%\end{equation}
%then 
%$$|c_M(t)-c_N(t)| \ge |B^{-1}|^{-1}\cdot |c_N(t)-\tilde c_0|\ge 2,$$
%i.e. the configuration of $M$ and $N$ is admissible; this proves (b). 
%
%
%Since $|B|^{-1}\le |B^{-1}|,$ we have 
%$D\subset D'.$ 

%
%\begin{remark} {\rm Lemma \ref{lm:61} does not give an answer to the question what happens when the centre $c_N(t)$ lies in the complement $D'-D$.  
%From (\ref{eq:mn2}) we see that a non-admissible configuration $(c_M(t), c_N(t))$ happens iff the vector $c_N(t)-\tilde c_0$ belongs to the image of the ball of radius $2$ under the operator $B^{-1}:\R^2\to \R^2$. The domain 
%$$
%\mathcal D=\{B^{-1}(x)+\tilde c_0; |x|\le 2\}
%$$
%satisfies $D\subset \mathcal D\subset D'$ and, complementing the statement of Lemma \ref{lm:61}, one may state that {\it the vehicles $M$ and $N$ are in collision iff $c_N(t)\in \mathcal D$}. The statement of Lemma \ref{lm:61} gives simple and easily verifiable criteria for absence or presence of collisions. 
%}
%\end{remark}
\begin{corollary}\label{cor:62}
If the autonomous vehicle $M$ employs a linear infinitesimal reaction mechanism with constant coefficients (\ref{eq:irm1}) and $(\alpha, \beta)\not= (1, 0)$ 
then \lq\lq the other vehicle\rq\rq\ $N$ is always capable of {\it \lq\lq forcing\rq\rq}\ a collision. 
\end{corollary}
\begin{proof}
Indeed, the vehicle $N$ can simply move into the ball $D\subset\mathcal D$ as we impose no restrictions on the motion of $N$. 
\end{proof}
\subsection{} To illustrate the above discussion we consider now a specific example of a linear infinitesimal reaction mechanism (\ref{eq:irm1}) with $\beta=0$ and $\alpha\not=1$, $\alpha>0$.  Thus we have $v_M=\alpha\cdot v_N$. Then the matrix $B$ has the form
$$B=\left[ 
\begin{array}{cc}
\alpha-1 & 0\\
0&\alpha-1
\end{array}
\right],\quad  B^{-1}=\left[ 
\begin{array}{cc}
(\alpha-1)^{-1} & 0\\
0&(\alpha-1)^{-1}
\end{array}
\right]$$ and, according to (\ref{eq:c0}),
\begin{equation}\label{eq:49a}
\tilde c_0=\frac{\alpha}{\alpha-1}\cdot c_N(0) - \frac{1}{\alpha-1}\cdot c_M(0).
\end{equation}
We see that $\tilde c_0$ lies on the straight line connecting $c_M(0)$ and $c_N(0)$. The matrix $B$ is diagonal and hence 
$$|B|^{-1}=|B^{-1}|=(\alpha-1)^{-1}.$$
The discs $D$ and $D'$ coincide with each other and with the domain $\mathcal D$ in this special case. 

If $\alpha >1$ then $c_N(0)$ lies between $c_M(0)$ and $\tilde c_0$, see Figure \ref{fig:2a} (left). 
\begin{figure}[h]
\begin{center}
\includegraphics[scale=0.5]{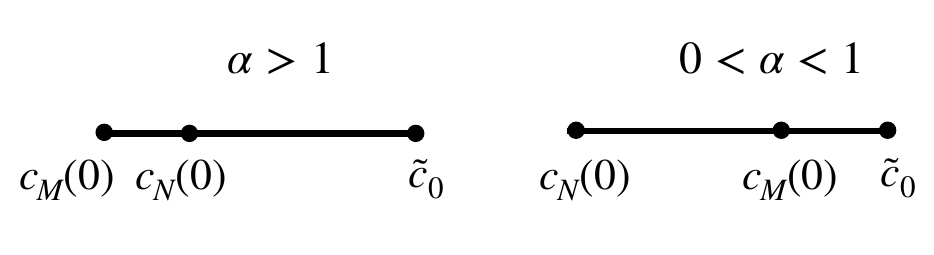}
\caption{The cases $\alpha>1$ (left) and $0<\alpha<1$ (right).}
\label{fig:2a}
\end{center}
\end{figure}
If $N$ moves towards $\tilde c_0$ with velocity $v_N$ then, according to the motion algorithm (\ref{eq:irm1}), $M$ moves towards $\tilde c_0$ with velocity 
$v_M=\alpha\cdot v_N.$ From (\ref{eq:49a}) we get
\begin{eqnarray}
|\tilde c_0-c_M(0)| =\alpha\cdot |\tilde c_0-c_N(0)|,
\end{eqnarray}
and hence, if $M$ employs the motion algorithm $v_M=\alpha\cdot v_N$, the centres $c_M(t)$ and $c_N(t)$ arrive at the point $\tilde c_0$ at the same time. 
Thus, for 
$\alpha >1$ 
the distance between $c_M$ and $c_N$ becomes smaller than $2$ before $N$ arrives to $\tilde c_0$. 

If $0<\alpha <1$ (see Figure \ref{fig:2a} right) then the similar analysis shows that the distance between the centres of $N$ and $M$ becomes smaller than $2$ before $M$ arrives to $\tilde c_0$.

The important difference between the cases $\alpha>1$ and $0<\alpha <1$ is that in the case $\alpha>1$ one can view the ego vehicle $M$ as being \lq\lq responsible\rq\rq\ for the collision,  
while in the case $0<\alpha <1$ \lq\lq the responsibility\rq\rq\ is on $N$ since it hits $M$ from behind. 

{\it Conclusions: } 

(1) As stated in Corollary \ref{cor:62}, the only linear reaction mechanism with constant coefficients (\ref{eq:irm1}) which guarantees safety (collision avoidance) 
is the case $(\alpha, \beta)=(1, 0)$, which is the reaction mechanism described in \S \ref{sec:53}. 

(2) This discussion illustrates the necessity to use {\it the reaction forms} and {\it the actuation functions} as described in Corollary \ref{cor:45} and in Example \ref{ex:46}. 
%\newpage
\section{Appendix: Proof of Lemma \ref{lm:6}}

Lemma \ref{lm:6} is a special  case of a more general statement which we state below as Proposition \ref{prop:app}. We start with a few general remarks and introduce the notations. 

Let $p: E\to B$ be a Hurewicz fibration and let $(K, L)$ be a pair consisting of a finite simplicial complex $K$ and its subcomplex $L$. 
The term \lq\lq a map\rq\rq\  always means  \lq\lq a continuous map\rq\rq. The notation $E^K$ stands for the space of all maps $K\to E$ equipped with the compact-open topology.  The notations $E^L$, $B^L$ and $B^K$ should be understood similarly. 
One has four associated maps
\begin{eqnarray}\label{diag:1}
\xymatrix{
E^K\ar[r]^{q_E}\ar[d]_{p^K}&E^L\ar[d]^{p^L}\\
B^K\ar[r]_{q_B}&B^L,
}
\end{eqnarray}
where the vertical maps $p^K$ and $p^L$ are compositions with the map $p$ and the horizontal maps $q_E$ and $q_B$ are restrictions 
on $L$ of the maps defined on $K$. The maps $q_E$ and $q_B$ are Hurewicz fibrations by Theorem 2 from \cite{Sp}, chapter 2, \S 8. The maps $p^K$ and $p^L$ are also Hurewicz fibrations, see \cite{Sp}, Exercise E.3 from chapter 2. 
%Thus all maps in diagram (\ref{diag:1}) are Hurewicz fibrations. 

Consider the pullback diagram 
\begin{eqnarray}\label{diag:2}
\xymatrix{
B^K\times_{B^L}E^L \ar[r]^{\hskip 1cm \pi_1}\ar[d]_{\pi_2} &E^L\ar[d]^{p^L}\\
B^K\ar[r]_{q_B}&B^L.
}
\end{eqnarray}
Recall that the space $B^K\times_{B^L}E^L$ is defined as the subset of the product $B^K\times E^L$ consisting of all pairs $\alpha: K\to B$, $\beta:L\to E$ such that $\alpha|_L =p\circ \beta$. Using diagram (\ref{diag:1}) we obtain a map 
\begin{eqnarray}\label{eq:ek}
\Pi: E^K\to B^K\times_{B^L}E^L 
\end{eqnarray}
where $\pi_1\circ \Pi=q^E$ and $\pi_2\circ \Pi=p^K$. 

\begin{proposition}\label{prop:app}
The  map (\ref{eq:ek}) 
is a Hurewicz fibration. 
\end{proposition}
\begin{proof} We show below that the map (\ref{eq:ek}) satisfies the homotopy lifting property. 

Let $X$ be a topological space and let $f: X\times I\to B^K\times_{B^L}E^L$ and $g: X\times\{0\}\to E^K$ be maps such that 
$$\Pi\circ g=f\circ i$$ where $i: X\times \{0\}\to X\times I$ is the inclusion. In other words, the following diagram commutes
$$
\xymatrix{
X\times\{0\}\ar[r]^g\ar[d]_i&E^K\ar[d]^{\Pi}\\
X\times I\ar[r]_{f\hskip 0.5 cm}\ar@{-->}[ru]^h&B^K\times_{B^L}E^L.
}
$$
We want to show the existence of a map $h: X\times I\to E^K$ such that 
\begin{eqnarray}\label{eq:want}
f=\Pi\circ h \quad \mbox{and}\quad  g=h\circ i. 
\end{eqnarray}

Using the exponential correspondence (see \cite{Sp}, Introduction, \S 2.8), we see that the compositions  $\pi_1\circ f$ and $\pi_2\circ f$ 
correspond to maps $f': X\times I\times L\to E$ and $f'': X\times I\times K\to B$ and  such that 
the diagram 
$$
\xymatrix{
X\times I\times L\ar[r]^{\hskip 0.5 cm f'}\ar[d]&E\ar[d]^p\\
X\times I\times K\ar[r]_{\hskip 0.5 cm f''}&B
}$$
commutes. Besides, the map $g$ corresponds to a map $g': X\times\{0\}\times K\to E$ which appears in the following commutative diagrams
$$
\xymatrix{
X\times \{0\}\times K\ar[r]^{\hskip 0.5 cm g'}\ar[d]_{i \times 1_K} &E\ar[d]^p\\
X\times I\times K\ar[r]_{\hskip 0.5 cm f''}&B,
}
\quad\quad \quad%\hskip 3cm
\xymatrix{
X\times\{0\}\times L\ar[d]_{i\times 1_L}\ar[r]^{\subset}&X\times\{0\}\times K\ar[d]^{g'}\\
X\times I\times L\ar[r]_{f'}& E.
}
$$
Here $i: X\times \{0\}\to X\times I$ denotes the inclusion. The diagram on the right shows that the maps $f'$ and $g'$ coincide on the intersection of their domains. 
Combining these two diagrams we obtain a commutative digram 
$$
\xymatrix{
{X\times I\times L\cup X\times\{0\}\times K}\ar[r]^{\hskip 2 cm f'\cup g'}\ar[d]_j&E\ar[d]^p\\
X\times I\times K\ar[r]_{\hskip 0.5 cm f''}&B.
}$$
Here $j$ is the natural inclusion and $f' \cup g'$ is the map determined by $f'$ and $g'$ as explained above.

By Corollary 3.2.4 from \cite{Sp} the space $X\times I\times L\cup X\times\{0\}\times K$ is a strong deformation retract of $X\times I\times K$. 
Using this fact and our assumption that $p:E\to B$ is a Hurewicz fibration we may apply the homotopy lifting property to the last diagram to obtain the existence of 
a map $h': X\times I\times K\to E$ with $$p\circ h'=f''\quad \mbox{and}\quad  h'\circ j=f' \cup g'.$$ 
By the exponential correspondence, $h'$ corresponds to a map 
$$h: X\times I\to E^K$$ satisfying 
\begin{eqnarray}
q_E\circ h=\pi_1\circ f,\quad p^K\circ h=\pi_2\circ f\quad \mbox{and}\quad h\circ i=g.
\end{eqnarray}
The first two equalities give $\Pi\circ h=f$; we see that both equalities (\ref{eq:want}) are satisfied. 
This completes the proof. 
\end{proof}
\begin{proposition}\label{prop:23}
The fibre of the fibration (\ref{eq:ek}) has homotopy type of the space of base-point preserving continuous maps $K/L\to F$, where $F$ is the fibre of 
$p:E\to B$. 
\end{proposition}
\begin{proof} Consider the pair $(\alpha, \beta)\in B^K\times_{B^L}E^L$ where $\alpha: K\to B$ and $\beta: L\to E$ are constant maps. 
The preimage $\Pi^{-1}(\alpha, \beta)$ is the space of all maps $f: K\to E$ with values in a single fibre $F\subset E$ of the fibration 
$p:E\to B$ such that 
$f(L)=\{e_0\}$ where $e_0\in F$ is the image of $\beta$. This proves our statement. 
Recall that all fibres of a Hurewicz fibration have the same homotopy type (see \cite{Sp}, Corollary 3.8.3).  
\end{proof}
\begin
{proof}[Proof of Lemma \ref{lm:6}] Lemma \ref{lm:6} is a special case of Proposition \ref{prop:app} combined with Proposition \ref{prop:23} 
with $K=I=[0,1]$ and $L=\{0, 1\}$. 
\end{proof}


\begin{thebibliography}{CFW21}



\bibitem{CFW21} D.C. Cohen, M. Farber and S. Weinberger, 
\textit{Topology of parametrized motion planning algorithms,}
SIAM Journal of Applied Algebra and Geometry, 5(2021), pp. 229-249. 

\bibitem{CFW22} D.C. Cohen, M. Farber and S. Weinberger, 
\textit{Parametrised topological complexity of collision-free motion planning in the plane,}
Annals of Mathematics and Artificial Intelligence, 
{\bf 90}(2022), pp. 999 --1015.



 \bibitem{E} Ch. Ehresmann, \textit{Les connexions infinit\'esimales dans un espace fibr\'e diff\'erentiable}, Colloque de Topologie, Bruxelles (1950), 29 - 55.

\bibitem{FGY} M. Farber, M. Grant and S. Yuzvinsky, \textit{Topological complexity of collision free motion planning algorithms in the presence of multiple moving obstacles}. Topology and robotics, 75–83, Contemp. Math., 438, Amer. Math. Soc., Providence, RI, 2007

\bibitem{FW23} M. Farber and S. Weinberger,
\textit{Parametrized motion planning and topological complexity,}
Proceedings of WAFR XV - "Workshop on Algorithmic Foundations of Robotics", 
Steven LaValle et al editors, Springer 2023, pp 1 - 17. 

 \bibitem{GC} {J. M. Garc\'{\i}a Calcines}. \textit{A note on covers defining relative and sectional categories.} Topol. Appl. 265, 106810
(2019).



\bibitem{Hoyo} M. del Hoyo, 
\textit{Complete connections on fiber bundles.} 
Indag. Math. (N.S.) 27 (2016), no. 4, 985–990. 




\bibitem{KMS} I. Kolář, P.W. Michor, J. Slovák, \textit{Natural operations in differential geometry.} Springer-Verlag, Berlin, 1993.

\bibitem{LP} K.M. Lynch and F.C. Park, \textit{Modern Robotics}, Cambridge University Press, 2019. 

\bibitem{SSS} S. Shalev-Shwartz, S. Shammah, A. Shashua, \textit{On a formal Model of Safe and Scalable Self-driving Cars}, arXiv:1708.06374v6




\bibitem{Sp} E. H. Spanier, \textit{Algebraic topology.}
Springer-Verlag, New York, 1995.


\bibitem{W} J.A. Wolf, \textit{Differentiable fibre spaces and mappings compatible with Riemannian metrics}, Michigan Math. J. 11 (1964), 65–70.




\end{thebibliography}
\end{document}